\documentclass{amsart}

\newif\ifULTRAMODE
\ULTRAMODEtrue

\usepackage{CJKutf8}
\usepackage[utf8]{inputenc} 
\usepackage{amsthm}
\usepackage{amsfonts}
\usepackage{scalerel}
\usepackage{amssymb}
\usepackage{ mathrsfs }
\usepackage{mathtools}
\usepackage{xcolor}
\usepackage{colonequals}
\usepackage{xfrac}
\usepackage{caption}
\usepackage{environ}
\usepackage{tikz}
\usepackage{adjustbox}
\usepackage{booktabs,tabularx}
\usepackage[normalem]{ulem}
\usepackage[shortlabels]{enumitem}
\usepackage{setspace}
\usepackage[hyphens]{url}
\usepackage{hyperref}
\usepackage[hyphenbreaks]{breakurl}
\usepackage[noabbrev]{cleveref}

\usepackage[citestyle=alphabetic,
style=alphabetic,
giveninits=true,
natbib=true,
backend=biber,
]
{biblatex}

\DeclareSourcemap{
  \maps[datatype=bibtex]{
    \map[overwrite=true]{
      \step[fieldset=url, null]
	  \step[fieldset=doi, null]
	  \step[fieldset=isbn, null]
	  \step[fieldset=issn, null]
    }
  }
}

\newcommand{\lang}[2]{\mscr{L}_{{\scriptscriptstyle #1}}^{{\scriptscriptstyle #2}}}

\newcommand{\compl}{{\mathsf{c}}}

\newcommand{\bb}{\mathbb}

\newcommand{\mcal}{\mathcal}
\newcommand{\mscr}{\mathscr}

\newcommand{\pow}{\mcal{P}}

\newcommand{\wo}{\setminus}
\newcommand{\msotimes}{\otimes}

\newcommand{\upperRomannumeral}[1]{\uppercase\expandafter{\romannumeral#1}}

\makeatletter

\newcommand{\Rmn}[1]{\expandafter\@slowromancap\romannumeral #1@}
\makeatother
\makeatletter
\newcommand{\oset}[3][0ex]{%
  \mathrel{\mathop{#3}\limits^{
    \vbox to#1{\kern-2\ex@
    \hbox{$\scriptstyle#2$}\vss}}}}
\makeatother
\newcommand{\exle}{\oset[0ex]{\scriptscriptstyle \exists}{<}}

\newcommand{\efi}{\forall}
\newcommand{\efii}{\exists}

\newcommand{\lift}[1]{\, \hat{#1} \,}

\newcommand{\conc}{\cdot}

\newcommand{\ew}{\varepsilon}

\newcommand{\FCon}{\Con_{\omega}}
\newcommand{\cnct}{\cdot}

\newcommand{\str}[1]{\mscr{#1}} 
\newcommand{\cla}[1]{\mscr{#1}} 
\newcommand{\cat}[1]{\mcal{#1}} 

\newcommand{\profin}[1]{\widehat{#1}} 

\DeclarePairedDelimiter\dangle{\langle}{\rangle}
\DeclareMathOperator{\Th}{Th}

\DeclareMathOperator{\qd}{qd}

\DeclareMathOperator{\PrimF}{PrimF}
\DeclareMathOperator{\UltraF}{UltraF}

\DeclareMathOperator{\BAlgO}{BAlgO}

\DeclareMathOperator{\BoolComp}{BoolComp}

\DeclareMathOperator{\TopMon}{TopMon}

\DeclareMathOperator{\At}{At}

\DeclareMathOperator{\Clop}{Clop}

\DeclareMathOperator{\Comp}{Comp}

\DeclareMathOperator{\Con}{Con}

\DeclareMathOperator{\MSO}{M}

\DeclareMathOperator{\LT}{LT}
\DeclareMathOperator{\FSpec}{SPEC}

\DeclareMathOperator{\base}{base}

\DeclareMathOperator{\dwcl}{dwcl}

\DeclareMathOperator{\Rec}{Rec}

\newtheorem{theorem}{Theorem}[section]
\newtheorem{lemma}[theorem]{Lemma}
\newtheorem{proposition}[theorem]{Proposition}

\theoremstyle{definition}
\newtheorem{definition}[theorem]{Definition}
\newtheorem{example}[theorem]{Example}

\newtheorem{corollary}[theorem]{Corollary}

\theoremstyle{remark}
\newtheorem{remark}[theorem]{Remark}
\newtheorem{notation}[theorem]{Notation}

\numberwithin{equation}{section}

\setlist[enumerate,itemize]{leftmargin=0.75cm,itemsep=1.5pt,topsep=1.5pt}

\begin{document}

\title[The pseudofinite mso theory of words, Deacon Linkhorn]{The pseudofinite monadic second order theory of words}

\author[The pseudofinite mso theory of words, Deacon Linkhorn]{Deacon Linkhorn}
\address{The University of Manchester, School of Mathematics, Oxford Road, Manchester, M13 9PL, UK}
\curraddr{}
\email{deacon.linkhorn@manchester.ac.uk}
\urladdr{https://deaconlinkhorn.com}
\thanks{}



\date{\today}

\dedicatory{}

\begin{abstract}
We analyse the pseudofinite monadic second order theory of words over a fixed finite alphabet. In particular we present an axiomatisation of this theory, working in a one-sorted first order framework. The analysis hinges on the fact that concatenation of words interacts nicely with monadic second order logic. More precisely, give a signature under which for each natural number k, equivalence of (monadic second order versions of) words with respect to formulas of quantifier depth at most k is a congruence for concatenation. 
We use our analysis to present an alternative proof of a theorem connecting recognisable languages and finitely generated free profinite monoids via extended Stone duality, due to Gehrke, Grigorieff, and Pin.
\end{abstract}

\maketitle
\tableofcontents

\section{Introduction}

In this paper we analyse the pseudofinite monadic second order theory of $\Sigma$-words for a fixed finite alphabet $\Sigma$, by which we mean the shared monadic second order theory of finite $\Sigma$-words. 
A (finite) $\Sigma$-word is a function $w:\alpha_w \rightarrow \Sigma$ where $\alpha_w$ is a (finite) linear order.
In particular we give an axiomatisation $T_{\MSO(\Sigma^*)}$ (\cref{Def-TMSigma}) of the pseudofinite monadic second order theory of $\Sigma$-words, working in a one-sorted first order framework.
To capture the monadic second order properties of a finite word $w:\alpha_w \rightarrow \Sigma$, we work with a first order structure $\MSO(w)$ in a one-sorted signature $\lang{\Sigma}{\msotimes}$ (\cref{Def-MSOVersion}). 
The structure $\MSO(w)$, which we refer to as the monadic second order version of $w$, is in particular an expansion of the Boolean algebra of subsets of $\alpha_w$. 
The analysis hinges on the fact that concatenation of finite words interacts nicely with monadic second order logic, in particular when working in the signature $\lang{\Sigma}{\msotimes}$ (\cref{Def-MSOVersion}).
More precisely, we make use of the fact that for each $k \in \bb{N}$ equivalence of (monadic second order versions of) words with respect to $\lang{\Sigma}{\msotimes}$-formulas of quantifier depth at most $k$ is a congruence for concatenation (\cref{Prop-kequivcongruence}). 
The central result of the paper is \cref{Theorem-TMSigmaCorrect}, in which we prove that the theory $T_{\MSO(\Sigma^*)}$ which we introduced is indeed an axiomatisation for the pseudofinite monadic second order theory of $\Sigma$-words.

As an application, we present an alternative proof of a theorem connecting recognisable languages and finitely generated free profinite monoids via extended Stone duality, due to Gehrke, Grigorieff, and Pin (\cref{Theorem-GGP}).
Our proof is relatively self-contained and any duality-theoretic notions required are introduced in a minimal fashion as and when they are needed.

\ \\
\noindent \textbf{Related literature:}
A similar approach to studying the logical properties of words, making use of equivalence up to bounded quantifier depth being a congruence for concatenation, is used in \cite{vGSteinberg} to explore the connection between star-free languages (a special collection of recognisable languages) and free pro-aperiodic monoids (a special collection of profinite monoids). 
The main difference between the two papers is that in \cite{vGSteinberg} it is first order logic on finite words, not monadic second order logic, which aligns with star-free langauges and free pro-aperiodic monoids.

In the author's PhD thesis \cite{DLPhD} an axiomatisation is given for the shared monadic second order theory of finite words over a singleton alphabet, i.e. the shared monadic second order theory of finite linear orders.
The connection between recognisable languages over this alphabet and the free profinite monoid on one generator, which arises as a special case of the theorem of Gehrke, Grigorieff, and Pin (\cref{Theorem-GGP}, hereafter referred to as the GGP theorem), was also considered in the thesis.
There is a significant difference in perspective between the thesis work and the work presented here. 
In the thesis the GGP theorem was taken as a foundation, while the logical analysis of the shared monadic second order theory of finite linear orders was seen as offering a new perspective on this theorem.
In this paper this viewpoint is inverted, we show that a logical analysis of the pseudofinite monadic second order theory of words in fact facilitates an alternative proof of the GGP theorem.

\ \\
\noindent Our logical notation, wherever it is not explicitly defined or referenced, is borrowed from \cite{HodgesBible}.

\section{Preliminaries}\label{SecPrelim}

\subsection{Boolean algebras and extended Stone duality}\label{Subsection-ExtendedStoneDuality}\ \\
Throughout the paper we assume that the reader is comfortable switching between viewing Boolean algebras as partial orders and as structures for the signature $\{\vee,\wedge,\bot,\top,^{\compl}\}$.

The following material on extended Stone duality is adapted from \cite{Hansoul} which works in the more general setting of spectral spaces and distributive lattices.
We also confine our attention to the case of binary operations and ternary relations, whereas the general theory allows for consideration of $n$-ary operations and $n+1$-ary relations.

\begin{definition}\label{Def-BAlgO}
Let $L$ be a Boolean algebra. 
A binary operation $\cdot:L^2 \rightarrow L$ is called,
\begin{enumerate}
\item \textbf{normal} if for each $a \in L$, $a \cdot \bot = \bot \cdot a = \bot$,
\item \textbf{additive} if for each $a,b,c \in L$, $(a \vee b) \cdot c = (a \cdot c) \vee (b \cdot c)$ and $a \cdot (b \vee c) = (a \cdot b) \vee (a \cdot c)$.
\end{enumerate}
We view Boolean algebras with normal additive operations as structures for the signature $\lang{}{} = \{\vee,\wedge,\bot,\top, ^{\compl},\cdot\}$.
We define $\BAlgO$ to be the category whose objects are Boolean algebras equipped with a normal additive binary operation, and whose morphisms are the usual homomorphisms for the signature $\lang{}{}$.
\end{definition}

\begin{definition}
Let $X$ be a Boolean space. 
A relation $R \subseteq X^3$ is called \textbf{compatible} if for each pair $(A,B) \in \Clop(X)^2$ the set,
\[
R(A,B,-) \coloneqq \{z \in X: \exists x \in A \exists y \in B \ R(x,y,z)\},
\]
is clopen in $X$.
In other words, $R \subseteq X^3$ is compatible if the the function $\pow(X)^2 \rightarrow \pow(X)$ defined by $(A,B) \mapsto R(A,B,-)$ restricts to $\Clop(X)$.
Given a ternary relation $R \subseteq X^3$ and $z \in X$ we write $R^{-1}(z)$ as shorthand for the set $\{(x,y) \in X^2: R(x,y,z)\}$.
We define $\BoolComp$ to be the category whose objects are Boolean spaces equipped with a compatible ternary relation, and whose morphisms $f:(X,R) \rightarrow (Y,S)$ are continuous functions $f:X \rightarrow Y$ such that for each $z \in X$,
\[
S^{-1}(f(z)) = f(R^{-1}(z)).
\]
\end{definition}

\cite{Hansoul} Theorem 1.14 presents a more general version of the following theorem.
 
\begin{theorem}\label{Theorem-ExtendedStoneDuality}
The categories $\BAlgO$ and $\BoolComp$ are dual to one another. 

Let $F$ be the functor $\BAlgO \rightarrow \BoolComp$ which for objects takes $(L,\cdot)$ to $(\ifULTRAMODE{\UltraF(L)}\fi,R_{\cdot})$ where,
\begin{enumerate}
\item \ifULTRAMODE{$\UltraF(L)$ }\fi is the Boolean space with underlying set comprising \ifULTRAMODE{ultrafilter}\fi s of $L$ and with the topology having an open basis comprising sets $D(a) \coloneqq \{p \in \ifULTRAMODE{\UltraF(L)}\fi: a \notin p\}$ where $a \in L$,
\item $R_{\cdot} \coloneqq \{(p,q,r) \in \ifULTRAMODE{\UltraF(L)}\fi^3 : \forall a \in p, \forall b \in q (a \cdot b \in r)\}$,
\end{enumerate}
and for morphisms takes $f:L_1 \rightarrow L_2$ in $\BAlgO$ to its preimage map $f^{-1}:\ifULTRAMODE{\UltraF(L_2)}\fi \rightarrow \ifULTRAMODE{\UltraF(L_1)}\fi$.

Let $G$ be the functor $\BoolComp \rightarrow \BAlgO$ which for objects takes $(X,R)$ to $(\Clop(X),\cdot)$ where,
\begin{enumerate}
\item $\Clop(X)$ is the Boolean algebra of clopen subsets of $X$,
\item $\cdot:\Clop(X)^2 \rightarrow \Clop(X)$ is the operation given by $A \cdot B \coloneqq \{z \in X: \exists x \in A, \exists y \in B \; R(x,y,z)\}$,
\end{enumerate}
and for morphisms takes $f:X_1 \rightarrow X_2$ in $\BoolComp$ to its preimage map $f^{-1}:\Clop(X_2) \rightarrow \Clop(X_1)$.

Then $FG$ is naturally isomorphic to the identity functor on $\BoolComp$ and $GF$ is naturally isomorphic to the identity functor on $\BAlgO$. 
\end{theorem}

\subsection{Monoids}\ \\
Our general setup for working with topological monoids and category-theoretic notions for the most part closely follows \cite{GehrkeSTR} (in particular section 4.2), however our notation slightly differs in places.
Most of the setup for this subsection can be applied (and in \cite{GehrkeSTR} section 4.2, is applied) in the more general setting of universal algebra, but we will confine our attention to the case of monoids. 

\begin{notation}\label{Notation-complexification}
Let $X$ be a set and $\cdot:X^2 \rightarrow X$ a binary operation. 
We write $\lift{\cdot}$ for the binary operation on $\pow(X)$ given by,
\[
A \lift{\cdot} B \coloneqq \{a \cdot b: a \in A,b \in B\},
\]
for each $A,B \in \pow(X)$.
We call $\lift{\cdot}$ the \textbf{complexification} of $\cdot$ (see \cite{Goldblatt}, where algebras of complexes are discussed, for an explanation of the terminology).
\end{notation}

\begin{definition}\label{Def-Congruence}
A \textbf{congruence} on a monoid $(M,\cdot)$ is an equivalence relation $\rho$ on the underlying set $M$ such that for any $a,b,a',b' \in M$, $(a,a'),(b,b') \in \rho$ implies $(a \cdot b,a' \cdot b') \in \rho$.
In other words, a congruence is an equivalence relation $\rho$ for which we get a well-defined binary operation on equivalence classes `pointwise' (i.e. taking $[a]_{\rho} \cdot [b]_{\rho} \coloneqq [a \cdot b]_{\rho}$).
Note that for each congruence $\rho$ on $M$ the set of equivalence classes $\sfrac{M}{\rho}$ forms a monoid under this `pointwise' operation. 

We write $\Con(M)$ for the collection of congruences on a monoid $M$, and $\Con_{\omega}(M)$ for the subcollection of finite index congruences, i.e. those with finitely many equivalence classes. 
For $\rho \in \Con(M)$, the homomorphism $\pi_{\rho}: M \rightarrow \sfrac{M}{\rho}$ given by $a \mapsto [a]_{\rho}$ is the called the \textbf{canonical quotient map} of $\rho$.
\end{definition}

\begin{definition}\label{Def-Refinement}
Let $X$ be a set and take two equivalence relations $\rho$ and $\theta$ on $X$.
We say that $\rho$ \textbf{refines} $\theta$, or equivalently that \textbf{$\theta$ is a coarsening of $\rho$}, precisely when $\rho \subseteq \theta$, in other words when $x \sim_{\rho} x'$ implies $x \sim_{\theta} x'$.
If $\rho$ and $\theta$ are equivalence relations on a set $X$, and $\rho$ refines $\theta$, then there is a unique function $p_{\rho,\theta}: \sfrac{X}{\rho} \rightarrow \sfrac{X}{\theta}$ such that $p_{\rho,\theta} \circ \pi_{\rho} = \pi_{\theta}$. The function $p_{\rho,\theta}$ is given by $p_{\rho,\theta}([x]_{\rho}) \coloneqq [x]_{\theta}$, and is called the \textbf{projection map}.
\end{definition}

\begin{definition}
A subset $S$ of a poset $(P,\leq)$ is called 
\begin{enumerate}[-]
\item \textbf{directed} if for each $a,b \in S$ there exists $c \in S$ such that $a \leq c$ and $b \leq c$ (note in some parts of the literature this would instead be called upwards directed),
\item \textbf{cofinal} if for each $a \in P$ there is $b \in S$ such that $a \leq b$.
\end{enumerate}
\end{definition}

\begin{remark}
For each monoid $M$, viewing $\Con(M)$ as a partial order under $\supseteq$, the subcollection $\FCon(M)$ is a directed subset.
In particular for any $\rho,\theta \in \FCon(M)$ the set-theoretic intersection $\rho \cap \theta$ is an element of $\FCon(M)$ with $\rho \supseteq \rho \cap \theta$ and $\theta \supseteq \rho \cap \theta$.
\end{remark}

\begin{definition}
The category $\TopMon$ of topological monoids has as objects triples $(M,\cdot,\tau)$ comprising a monoid $(M,\cdot)$ and a topology $\tau$ on $M$ such that $\cdot:M^2 \rightarrow M$ is continuous ($M^2$ is given the product topology).
Arrows in $\TopMon$ are functions $f:M \rightarrow M'$ which are monoid homomorphisms that are moreover continuous with respect to the topologies. 
\end{definition}

\begin{definition}
A \textbf{directed system of topological monoids} is (for us) a \emph{contravariant} functor $F:I \rightarrow \TopMon$ where $I$ is a directed partial order viewed as a category.

Recall that a \textbf{cone} for a (contravariant) functor $F:\cat{C} \rightarrow \cat{D}$ comprises an object $d \in \cat{D}$ and for each $c \in \cat{C}$ an arrow $f_c:d \rightarrow F(c)$, such that for each arrow $f:c' \rightarrow c$ in $\cat{C}$ the equality $F(f) \circ f_c = f_d$ holds.

The \textbf{inverse limit} (also known as projective limit) of a directed system of topological monoids $F:I \rightarrow  \TopMon$ is a cone $(M,(\profin{\pi}_i)_{i \in I})$ for $F$ 
satisfying the universal property that for any other cone $(N,(g_i)_{i \in I})$ for $F$ there is a \emph{unique} morphism $h:N \rightarrow M$ such that for each $i \in I$ we have $g_i = \profin{\pi}_i \circ h$.
Note that unlike direct limits, which are in fact colimits, inverse limits really are limits according to the category theory lexicon. 
The inverse limit of the directed system $F$, which always exists in the category $\TopMon$, is denoted by $\varprojlim F$.
\end{definition}

\begin{lemma}\label{Lemma-CofinalDirectedSystem}
Let $F:I \rightarrow \TopMon$ be a directed system of topological monoids. 
If $C \subseteq I$ is a cofinal subset then the restriction $F'$ of $F$ to $C$ is a directed system of topological monoids and the two topological monoids $\varprojlim F$ and $\varprojlim F'$ are isomorphic.
\end{lemma}
\begin{proof}
See for example \cite{Mac}, in particular Theorem 1 on page 217 (which works in a more general setting and under slightly different notation and setup).
\end{proof}

\begin{definition}\label{Def-Profinite}
Let $M$ be a monoid. 
Let $F_M:(\Con_{\omega}(M),\supseteq) \rightarrow \TopMon$ be the directed system given by taking $F_M(\rho)$ to be the finite monoid $\sfrac{M}{\rho}$ endowed with the discrete topology, and taking the projection map $p_{\rho,\theta}:\sfrac{M}{\rho} \rightarrow \sfrac{M}{\theta}$ (see \cref{Def-Refinement}) whenever $\theta \supseteq \rho$.
(Note this is a contravariant functor, as $F_M$ sends an arrow $\rho \xleftarrow{\subseteq} \theta$ to an arrow $\sfrac{M}{\rho} \xrightarrow{p_{\rho,\theta}} \sfrac{M}{\theta}$.)
We write $\profin{M}$ for the inverse limit $\varprojlim F_M$.
The topological monoid $\profin{M}$ is called the \textbf{profinite completion} of $M$.
In the case where $M$ is the free monoid $\Sigma^*$ on the set $\Sigma$, the topological monoid $\profin{\Sigma^*}$ is called the \textbf{free profinite monoid on $\Sigma$}.
In general a topological monoid $(M,\cdot,\tau)$ is called \textbf{profinite} if there exists a directed system $F:I \rightarrow \TopMon$ such that for each $i \in I$ the topological monoid $F(i)$ is finite and carries the discrete topology, such that $(M,\cdot,\tau) = \varprojlim F$.
\end{definition}

\begin{remark}
The choice of `profinite completion' as terminology for the inverse limit $\profin{M}$ may seem strange at first.
This choice can be justified in multiple ways. 
In a very general setting, as is noted in \cite{GehrkeSTR} pp 28, the forgetful functor from the category of profinite topological algebras of a given type into the category of algebras of that type admits a left adjoint, and this left adjoint sends an algebra $M$ to its profinite completion.
Alternatively, in the case of $\Sigma^*$ there exist `profinite metrics' (the distance between two words is inversely proportional to the minimal index of a finite congruence on $\Sigma^*$ for which the two words reside in distinct equivalence classes) for which $\profin{\Sigma^*}$ is the metric completion.
For the latter approach, see for example \cite{JEPinProfin}.
\end{remark}

\subsection{Uppercase/lowercase notation and comprehension schema}\ \\
We distinguish between upper and lower case variables as a notational shorthand for a certain relativisation of formulas (see \cite{HodgesBible} Section 5.1 pp 202 for a general overview of the relativisation technique) in the setting of atomic Boolean algebras.

\begin{definition}
An \textbf{atom} of a partial order $(P,\leq)$ with a smallest element $\bot$ is a minimal element of $P \wo \{\bot\}$. 
A Boolean algebra, viewed as a partial order, is called \textbf{atomic} if every element lies above an atom.
Note that for any partial order $(P,\leq)$ the atoms form a $0$-definable subset of $P$, and moreover this can be done uniformly.
By uniformly we mean that there is a single $\lang{}{} = \{\leq\}$ formula $\At(x)$ which defines the set of atoms in each partial order $(P,\leq)$.
Later, for the purpose of eliminating quantifiers, we will simply include $\At$ as a unary predicate symbol in the signatures we work with.
\end{definition}

\begin{remark}
For a Boolean algebra $\str{B}$, being atomic is equivalent to the following holding, 
\[
\forall x \forall y ((\forall z (\At(z) \rightarrow (z \leq x \leftrightarrow z \leq y))) \rightarrow x = y). \tag{$\dagger$}\label{AtomicFormula}
\]
This says, if two elements sit above the same atoms then they are equal.
\end{remark}

\begin{notation}
To make formulae more easily readable when working over theories of (expansions of) atomic Boolean algebras we make use of some notational conventions. For atoms we reserve lower case letters $x,y,z,\ldots$ for variables, and lower case letters $a,b,c,\ldots$ for constants and parameters. For general elements we use upper case letters $X,Y,Z,\ldots$ for variables and upper case letters $A,B,C,\ldots$ for constants and parameters.
Formally, the use of a lower case variable in a formula is shorthand for relativisation of that variable to the $0$-definable set of atoms in that formula. 
Moreover, we write $X(x)$ as shorthand for $x \subseteq X$.

To illustrate with an example, 
\[
\forall X \forall Y (\forall z (X(z) \leftrightarrow Y(z)) \rightarrow X=Y),
\]
is the rewriting of (\ref{AtomicFormula}) using these notational shorthands.
\end{notation}

\begin{definition}\label{Def-Comprehension}
Let $\lang{}{}$ be a signature containing a binary relation symbol $\subseteq$.

For each $\lang{}{}$-formula $\phi(x,\bar{Y})$, we define $\Comp_{\phi}$ to be the $\lang{}{}$-sentence,
\[
\forall \bar{Y} \exists Z \forall x (Z(x) \leftrightarrow \phi(x,\bar{Y})).
\]
(Recall that $Z(x)$ is shorthand for $x \subseteq Z$.)

If $\str{M}$ is an $\lang{}{}$-structure in which $\subseteq$ is the ordering of an atomic Boolean algebra, then $\Comp_{\phi}$ says that for each tuple of parameters $\bar{Y}$ there is an element $Z$ of the Boolean algebra which lies above precisely the atoms $x$ in the set defined by the formula $\phi$ with parameters $\bar{Y}$.

The \textbf{comprehension schema} for the signature $\lang{}{}$ comprises the sentence $\Comp_{\phi}$ for each $\lang{}{}$-formula $\phi$. 
\end{definition}

\section{Axiomatising the pseudofinite monadic second order theory of \texorpdfstring{$\Sigma$}{Σ}-words}\label{SecAxiomatisation}

We will make use of a one-sorted first-order setup to capture monadic second order structures. Space does not permit a detailed discussion of the motivations behind choosing this approach or the alternatives available. For these the reader is referred to the author's PhD thesis \cite{DLPhD}, in particular chapters 1 and 3. 

\begin{definition}
Let $\Sigma$ be a finite set. 
A \textbf{$\Sigma$-word} is a function $w:\alpha_w \rightarrow \Sigma$ where $\alpha_w$ is a linear order. 
We write $\Sigma^*$ for the collection of finite $\Sigma$-words, including the unique empty $\Sigma$-word $\ew:\emptyset \rightarrow \Sigma$.
For $\sigma \in \Sigma$ we sometimes identify $\sigma$ with the $\Sigma$-word $\{*\} \rightarrow \Sigma$, $* \mapsto \sigma$. 
\end{definition}

\begin{definition}\label{Def-MSOVersion}
We define $\lang{\Sigma}{\msotimes}$ to be the signature $\{\subseteq,\exle,\At,\bot,(P_{\sigma})_{\sigma \in \Sigma}\}$ in which,
\begin{enumerate}[-]
\item $\subseteq$ and $\exle$ are binary relation symbols,
\item $\At$ is a unary relation symbol,
\item $\bot$ is a constant symbol, 
\item $P_{\sigma}$ is a unary relation symbol for each $\sigma \in \Sigma$.
\end{enumerate}
For each $w:\alpha_w \rightarrow \Sigma$ from $\Sigma^*$ we define $\MSO(w)$ to be the $\lang{\Sigma}{\msotimes}$-structure in which,
\begin{enumerate}
\item the universe is $\pow(\alpha_w)$,
\item $\subseteq$ is the usual set-theoretic inclusion,
\item $\exle$ is $\{(A,B) \in \pow(\alpha)^2: a < b$ for some $a \in A, b \in B\}$, where $<$ is the ordering of $\alpha_w$,
\item $\At$ is the collection of atoms of $(\pow(\alpha_w),\subseteq)$, i.e. the singleton subsets of $\alpha_w$,
\item $\bot$ is the empty set,
\item for each $\sigma \in \Sigma$ the unary relation symbol $P_{\sigma}$ is given the interpretation,
\[
\{\{a\}: a \in w^{-1}(\sigma)\},
\]
i.e. $P_{\sigma}$ is a collection of atoms of the Boolean algebra, namely those $\{a\}$ for which $w(a) = \sigma$. 
\end{enumerate}
We call the $\lang{\Sigma}{\msotimes}$-structure $\MSO(w)$ the \textbf{monadic second order version of the word $w$}.
The \textbf{pseudofinite monadic second order theory of $\Sigma$-words} is defined to be the $\lang{\Sigma}{\msotimes}$-theory $\bigcap_{w \in \Sigma^*} \Th(\MSO(w))$.
We write $\Th(\MSO(\Sigma^*))$ as a shorthand for this theory.
\end{definition}

\begin{example}
Let $\Sigma = \{\sigma,\gamma,\delta\}$ and consider the $\Sigma$-word $\sigma \gamma \sigma$. 
In the setup we use this word is to be viewed as a function $w:\{0,1,2\} \rightarrow \{\sigma,\gamma\}$.
Then the monadic second order version $\MSO(w)$ comprises an expansion of the $8$ element Boolean algebra $\pow(\{0,1,2\})$.
In this expansion we have $P_{\sigma} = \{0,2\}$, $P_{\gamma} = \{1\}$ and $P_{\delta} = \emptyset$. 
\end{example}

Rather than give our axiomatisation of $\Th(\MSO(\Sigma^*))$ immediately, we first introduce a weaker $\lang{\Sigma}{\msotimes}$-theory $T_{\base(\Sigma)}$ which will serve as the setting for several technical lemmas. 

\begin{definition}
We define an $\lang{\Sigma}{\msotimes}$-theory $T_{\base(\Sigma)}$ to be the set of sentences expressing,
\newcounter{axiomnumbers}
\begin{enumerate}
\item atomic Boolean algebra (possibly trivial in the sense that $\bot = \top$) under $\subseteq$, with the atoms picked out by the unary relation $\At$,
\item $\bot$ is the bottom element of the Boolean algebra,
\item atoms linearly ordered (possibly trivially in the sense that the collection of atoms is empty) under $\exle$,
\item $\forall X,Y (X \exle Y \leftrightarrow \exists x \exists y (X(x) \wedge Y(y) \wedge x \exle y))$,
\item the relations $P_{\sigma}$ only hold for atoms and,
\[
\forall x (\bigvee_{\sigma \in \Sigma} P_{\sigma}(x) \wedge \bigwedge_{\substack{\sigma,\sigma' \in \Sigma \\ \sigma \neq \sigma'}} \neg (P_{\sigma}(x) \wedge P_{\sigma'}(x))).
\]
\setcounter{axiomnumbers}{\value{enumi}}
\end{enumerate}
\end{definition}

\begin{definition}\label{Def-TMSigma}
We define an $\lang{\Sigma}{\msotimes}$-theory $T_{\MSO(\Sigma^*)}$ to be the extension of $T_{\base(\Sigma)}$ by sentences expressing,
\begin{enumerate}
\setcounter{enumi}{\value{axiomnumbers}}
\item the linear ordering $\exle$ on the atoms, if not trivial, is discrete with endpoints (we denote the smallest and largest atoms by $0$ and $0^*$ respectively in this case),
\item each element not equal to $\bot$ lies above a smallest atom with respect to $\exle$,
\item the comprehension schema for $\lang{\Sigma}{\msotimes}$ (see \cref{Def-Comprehension}).
\end{enumerate}
\end{definition}

\begin{remark}
There is some redundancy in our choice of axioms. For example, the discreteness of the linear order and the presence of endpoints can both be seen to follow from the other axioms.
\end{remark}

For the remainder of this section, we prove that $T_{\MSO(\Sigma^*)}$ axiomatises the pseudofinite monadic second order theory of linear order.
Along the way some of the machinery required in \cref{Section-FreeProfin} will be introduced.

First we establish some lemmas working with our base theory $T_{\base(\Sigma)} \subseteq T_{\MSO(\Sigma^*)}$.

\begin{definition}
Let $(\alpha,<_{\alpha})$ and $(\beta,<_{\beta})$ be linear orders. We define $\alpha + \beta$ to be the linear ordering on the disjoint union $\alpha \sqcup \beta$ given by $<_{\alpha} \cup <_{\beta} \cup (\alpha \times \beta)$.
Given $\Sigma$-words $w:\alpha \rightarrow \Sigma$ and $v:\beta \rightarrow \Sigma$, we define their concatenation to be the $\Sigma$-word $w \conc v: \alpha + \beta \rightarrow \Sigma$ given by,
\[
(w \conc v) (i) \coloneqq
\begin{cases}
w(i) &\text{ if }i \in \alpha,\\
v(i) &\text{ if }i \in \beta.
\end{cases}
\]
\end{definition}

\begin{remark}
It is well known that the theory of Boolean algebras in the signatures $\{\subseteq\}$ and $\{\cup,\cap,^{\compl},\bot,\top\}$ are bi-interpretable.

Given Boolean algebras $\str{B}_1$ and $\str{B}_2$, the direct product $\str{B}_1 \times \str{B}_2$ given by taking the pointwise ordering is again a Boolean algebra. 
\end{remark}

\begin{definition}
Given $\lang{\Sigma}{\msotimes}$-structures $\str{M}$ and $\str{N}$, we define an $\lang{\Sigma}{\msotimes}$-structure $\str{M} \msotimes \str{N}$ as follows,
\begin{enumerate}
\item the universe of the structure is $\str{M} \times \str{N}$,
\item $\subseteq$ is defined as it is in the direct product, i.e. for all $(A,B),(C,D) \in \str{M} \msotimes \str{N}$ we declare $\str{M} \msotimes \str{N} \models (A,B) \subseteq (C,D)$ iff $\str{M} \models A \subseteq C$ and $\str{N} \models B \subseteq D$,
\item $\bot$ is interpreted in $\str{M} \msotimes \str{N}$ as $(\bot^{\str{M}},\bot^{\str{N}})$,
\item $\At$ is interpreted in $\str{M} \msotimes \str{N}$ as the collection of elements of the form $(A,\bot)$ where $A \in \At^{\str{M}}$ or $(\bot,B)$ where $B \in \At^{\str{N}}$,
\item $\exle$ is given  by $\str{M} \msotimes \str{N} \models (A,B) \exle (C,D)$ if and only if,
\[
\str{M} \not\models A = \bot \text{ and } \str{N} \not\models D = \bot, \text{ or }  \str{M} \models A \exle C, \text{ or }\str{N} \models B \exle D,
\]
\item for each $\sigma \in \Sigma$, $P_{\sigma}$ is given by $\str{M} \msotimes \str{N} \models P_{\sigma}((A,B))$ if and only if,
\[
\str{M} \models P_{\sigma}(A) \text{ and } \str{N} \models B = \bot, \text{ or, }\str{M} \models A = \bot \text{ and } \str{N} \models P_{\sigma}(B).
\]
\end{enumerate}
\end{definition}

\begin{remark}
For any $w,v \in \Sigma^*$ we have $\MSO(w) \msotimes \MSO(v) \cong \MSO(w \conc v)$.
We can therefore view $w \mapsto \MSO(w)$ as a homomorphism from $(\Sigma^*,\conc)$ to the class of $\lang{\Sigma}{\msotimes}$-structures with the operation $\msotimes$.

More generally, if $\str{M},\str{N} \models T_{\base(\Sigma^*)}$ then $\str{M} \msotimes \str{N}$ is the unique $\lang{\Sigma}{\msotimes}$-structure expanding the direct product of the Boolean algebras underlying $\str{M}$ and $\str{N}$, in which $\exle$ and the predicates $P_{\sigma}$ are interpreted according to the concatenation of the words underlying $\str{M}$ and $\str{N}$. 
\end{remark}

Working over a finite signature, the following definition organises unnested first-order formulas into a union of finite sets according to the amount of nesting of quantifiers involved. 
As each formula is equivalent over the empty theory to an unnested formula (\cite{HodgesBible} 2.6.2), this allows for inductive proofs targeting all formulas while only dealing with finitely many at a time. 

\begin{definition}[\cite{HodgesBible} pp. 103]\label{Definition-QuantifierDepth}
Recall that an $\lang{}{}$-formula is called \textbf{unnested} if all of its atomic subformulae are of the form $x = y$, $c=y$, $F(\bar{x}) = y$, or $R(\bar{x})$ where $x,y,\bar{x},\bar{y}$ comprise only variables, $c$ is a constant symbol of $\lang{}{}$, $F$ is a function symbol of $\lang{}{}$, and $R$ is a relation symbol of $\lang{}{}$.

Given an $\lang{}{}$-formula $\phi(\bar{x})$, the \textbf{quantifier depth} of $\phi$, denoted by $\qd(\phi)$, is defined by induction on the complexity of $\phi$ as follows,
\begin{enumerate}[-]
\item $\qd(\phi) = 0$ if $\phi$ is quantifier-free,
\item $\qd(\phi) = \max(\{\qd(\phi_i): 1 \leq i \leq n\})$ if $\phi$ is a Boolean combination of $\lang{}{}$-formulae $\phi_1,\ldots,\phi_n$,
\item $\qd(\exists v \phi) = \qd(\forall v \phi) = \qd(\phi)+1$.
\end{enumerate}

For each $k \in \bb{N}$ we can define an equivalence relation $\approx_k$ on the class of $\lang{}{}$-structures by taking $\str{M} \approx_k \str{N}$ if and only if for each \emph{unnested} $\lang{}{}$-sentence $\phi$ of quantifier depth at most $k$, $\str{M} \models \phi$ if and only if $\str{N} \models \phi$. 
\end{definition}

\begin{remark}
If $\lang{}{}$ is finite, then for each $k$ the equivalence relation $\approx_k$ on the class of $\lang{}{}$-structures admits finitely many equivalence classes.
This is because there are finitely many unnested sentences of quantifier depth at most $k$.
\end{remark}

\begin{theorem}[Fra\"iss\'e-Hintikka Theorem (\cite{HodgesBible} Theorem 3.3.2)]
Let $\lang{}{}$ be a finite first order signature. 
Let $k \in \bb{N}$, and write $C_1,\ldots,C_n$ for the finitely many equivalence classes of $\approx_k$. 
Then there are unnested $\lang{}{}$-sentences $\phi_1,\ldots,\phi_n$ of quantifier depth at most $k$, such that for each $\lang{}{}$-structure $\str{M}$ and each $i$ with $1 \leq i \leq n$,
\[
\str{M} \in C_i \Longleftrightarrow \str{M} \models \phi_i.
\]
We refer to $\phi_1,\ldots,\phi_n$ as \emph{Hintikka}-sentences.
\end{theorem}

\begin{remark}
Hintikka sentences axiomatising equivalence classes can be given effectively by an induction on $k$ but this is not required for our applications, see \cite{HodgesBible} 3.3.2 for more details. 

As each $\lang{}{}$-formula is equivalent to an unnested $\lang{}{}$-formula, it is clear that for two $\lang{}{}$-structures $\str{M}$ and $\str{N}$ the following are equivalent,
\begin{enumerate}
\item $\str{M} \equiv \str{N}$,
\item for each $k \in \bb{N}$, $\str{M} \approx_k \str{N}$.
\end{enumerate}

For each $k$, the equivalence relation $\approx_k$ can also be defined in terms of the second player having a winning strategy in the so-called \emph{unnested Ehrenfeucht-Fra\"iss\'e game} of length $k$, which we denote by $G_k$, between two $\lang{}{}$-structures. 
We make use of this equivalent formulation in the proof of the upcoming \cref{Prop-kequivcongruence} but space does not permit a full definition of these games. 
A full account of the games can be found in \cite{HodgesBible} pp. 102.
\end{remark}

\begin{proposition}\label{Prop-kequivcongruence}
Let $\str{M}_1,\str{M}_2,\str{N}_1,\str{N}_2 \models T_{\base(\Sigma)}$. Let $k \in \bb{N}$. 
If $\str{M}_1 \approx_k \str{N}_1$ and $\str{M}_2 \approx_k \str{N}_2$ then $\str{M}_1 \msotimes \str{M}_2 \approx_k \str{N}_1 \msotimes \str{N}_2$.
\end{proposition}
\begin{proof}
We outline a winning strategy of $\efii$ (the player who moves second) for the unnested Ehrenfeucht-Fra\"iss\'e game of length $k$ played on the structures $\str{M}_1 \msotimes \str{M}_2$ and $\str{N}_1 \msotimes \str{N}_2$, i.e. for $G_k(\str{M}_1 \msotimes \str{M}_2, \str{N}_1 \msotimes \str{N}_2)$.
The basic idea is to fuse moves taken from winning strategies for the games $G_k(\str{M}_1,\str{N}_1)$ and $G_k(\str{M}_2,\str{N}_2)$. 

More explicitly, the strategy of $\efii$ is as follows. 
We give the case where $\efi$ makes a move in $\str{M}_1 \msotimes \str{M}_2$, the case where $\efi$ moves in $\str{N}_1 \msotimes \str{N}_2$ is dealt with symmetrically.
\begin{enumerate}
\item Choose winning strategies in the games $G_k[\str{M}_1,\str{N}_1]$ and $G_k[\str{M}_2,\str{N}_2]$, possible by the assumption that $\str{M}_1 \approx_k \str{N}_1$ and $\str{M}_2 \approx_k \str{N}_2$,
\item upon receiving the $i$-th move $A_i = (A_{i1},A_{i2}) \in \str{M}_1 \msotimes \str{M}_2$ from $\efi$, take both $A_{i1}$ and $A_{i2}$ and feed them into simulations of the games $G_k[\str{M}_1,\str{N}_1]$ and $G_k[\str{M}_2,\str{N}_2]$ respectively,
\item take responses $B_{i1} \in \str{N}_1$ and $B_{i2} \in \str{N}_2$ dictated by the winning strategies for $G_k[\str{M}_1,\str{N}_1]$ and $G_k[\str{M}_2,\str{N}_2]$ respectively which were chosen at the outset,
\item form the move $B_i = (B_{i1},B_{i2}) \in \str{N}_1 \msotimes \str{N}_2$ as the response.
\end{enumerate}

Suppose that the strategy just described is deployed, and the resulting play of the game is $((A_1,B_1),\ldots,(A_k,B_k)) \in ((\str{M}_1 \msotimes \str{M}_2) \times (\str{N}_1 \msotimes \str{N}_2))^k$, with $A_i = (A_{i1},A_{i2}) \in \str{M}_1 \times \str{M}_2$ and $B_i = (B_{i1},B_{i2}) \in \str{N}_1 \times \str{N}_2$ for each $i$ such that $1 \leq i \leq k$.

We must check that for each unnested atomic $\lang{\Sigma}{\msotimes}$-formula $\phi(\bar{X})$ we have,
\[
\str{M}_1 \msotimes \str{M}_2 \models \phi(\bar{A}) \Leftrightarrow \str{N}_1 \msotimes \str{N}_2 \models \phi(\bar{B}).
\]
This can be done on a case by case basis.
Unnested atomic formulae in $\bar{X}$ are of the form $X = X'$, $X \subseteq X'$, $\At(X)$, $X \exle X'$, or $P_{\sigma}(X)$, where $X,X' \in \bar{X} \cup \{\bot\}$.

We deal with the case $X \exle X'$, and leave the remaining cases as an exercise for the reader. 
Suppose,
\[
\str{M}_1 \msotimes \str{M}_2 \models (A_{i1},A_{i2}) \exle (A_{j1},A_{j2}).
\] 
By definition of $\msotimes$ this is equivalent to at least one of the following conditions holding,
\begin{enumerate}
\item $\str{M}_1 \not\models A_{i1} = \bot$ and $\str{M}_2 \not\models A_{j2} = \bot$,
\item $\str{M}_1 \models A_{i1} \exle A_{j1}$, 
\item $\str{M}_2 \models A_{i2} \exle A_{j2}$.
\end{enumerate}
Now by our assumption that the pairs $(\bar{A}_1,\bar{B}_1)$ and $(\bar{A}_2,\bar{B}_2)$ are winning plays for $\efii$ in $G_k[\str{M}_1,\str{N}_1]$ and $G_k[\str{M}_2,\str{N}_2]$ respectively, these conditions are equivalent to the following respectively,
\begin{enumerate}
\item $\str{N}_1 \not\models B_{i1} = \bot$ and $\str{N}_2 \not\models B_{j2} = \bot$,
\item $\str{N}_1 \models B_{i1} \exle B_{j1}$, 
\item $\str{N}_2 \models B_{i2} \exle B_{j2}$.
\end{enumerate}
By definition of $\msotimes$, the disjunction of these conditions is equivalent to,
\[
\str{N}_1 \msotimes \str{N}_2 \models (B_{i1},B_{i2}) \exle (B_{j1},B_{j2}).
\]
By symmetry we can therefore conclude,
\[
\str{M}_1 \msotimes \str{M}_2 \models (A_{i1},A_{i2}) \exle (A_{j1},A_{j2}) \Leftrightarrow \str{N}_1 \msotimes \str{N}_2 \models (B_{i1},B_{i2}) \exle (B_{j1},B_{j2}).
\]
Hence we have outlined a winning strategy for $\efii$ in $G_k[\str{M}_1 \msotimes \str{M}_2,\str{N}_1 \msotimes \str{N}_2]$. 
\end{proof}

\begin{lemma}\label{Lemma-ClassTheory}
Let $\cla{C}$ be a class of $\lang{}{}$-structures.
For each $\lang{}{}$-structure $\str{M}$ the following are equivalent,
\begin{enumerate}[(1)]
\item $\str{M} \models \bigcap_{\str{N} \in \cla{C}} \Th(\str{N})$,
\item $\Th(\str{M}) \subseteq \bigcup_{\str{N} \in \cla{C}} \Th(\str{N})$.
\end{enumerate}
\end{lemma}
\begin{proof}
We write $\Th(\cla{C})$ for $\bigcap_{\str{N} \in \cla{C}} \Th(\str{N})$.

(1) implies (2):\\
Towards a contradiction suppose that $\str{M} \models \Th(\cla{C})$ holds while $\Th(\str{M}) \subseteq \bigcup_{\str{N} \in \cla{C}} \Th(\str{N})$ fails.
Then there exists a sentence $\phi \in \Th(\str{M}) \wo \bigcup_{\str{N} \in \cla{C}} \Th(\str{N})$. 
But now note that for any sentence $\phi$, $\phi \notin \bigcup_{\str{N} \in \cla{C}} \Th(\str{N})$ implies that $\neg \phi \in \bigcap_{\str{N} \in \cla{C}} \Th(\str{N}) = \Th(\cla{C})$.
Therefore we get that $\str{M} \models \phi \wedge \neg \phi$, giving us a contradiction.

(2) implies (1):\\
Towards a contradiction suppose that $\Th(\str{M}) \subseteq \bigcup_{\str{N} \in \cla{C}} \Th(\str{N})$ while $\str{M} \not\models \Th(\cla{C})$. 
Then there exists a sentence $\phi \in \Th(\cla{C}) \wo \Th(\str{M})$.
Therefore $\neg \phi \in \Th(\str{M}) \subseteq \bigcup_{\str{N} \in \cla{C}} \Th(\str{N})$.
This implies that for some $\str{N} \in \cla{C}$ we have $\str{N} \models \phi \wedge \neg \phi$, giving us a contradiction.
\end{proof}

\begin{remark}
\Cref{Lemma-ClassTheory} can be rephrased into a more order-theoretic presentation. 
If $\str{B}$ is a Boolean algebra, $\{\mcal{V}_i:i \in I\}$ is a collection of ultrafilters of $\str{B}$, and $\mcal{U}$ is an ultrafilter of $\str{B}$, then the following are equivalent,
\begin{enumerate}[(1)]
\item $\bigcap_{i \in I} \mcal{V}_i \subseteq \mcal{U}$,
\item $\mcal{U} \subseteq \bigcup_{i \in I} \mcal{V}_i$.
\end{enumerate}
It is well known that each Boolean algebra is isomorphic to the Lindenbaum-Tarski algebra of some theory $T$ in some signature $\lang{}{}$ (see forthcoming \cref{Def-Lindenbaum}).
Ultrafilters of the Lindenbaum-Tarski algebra of a theory $T$ are precisely complete extensions of $T$, i.e. theories of the form $\Th(\str{M})$ for some $\str{M} \models T$.
Therefore the order-theoretic presentation is equivalent to the logical presentation \cref{Lemma-ClassTheory}.
\end{remark}

\begin{definition}
Let $\str{M} \models T_{\MSO(\Sigma^*)}$. 
For each $a \in \At(\str{M})$ it follows immediately from the comprehension scheme (in particular from $\Comp_{\phi}$ for $\phi(x,a): x \exle a \vee x = a$) that the initial interval induced by $a$ is an element of $\str{M}$. 
We denote this element by $[0,a]$.
Note that $[0,a]$ is in the definable closure of the element $a$.
\end{definition}

\begin{lemma}\label{Lemma-ResRel}
Let $\str{M} \models T_{\MSO(\Sigma^*)}$.
For each $a \in \At(\str{M})$, the set of elements which are beneath the element $[0,a]$ with respect to the ordering $\subseteq$ form the underlying set of a substructure of $\str{M}$. We denote this substructure by $\str{M} \upharpoonright [0,a]$.

By relativisation of quantifiers (see \cite{HodgesBible} Section 5.1 pp 202), for each $\lang{\Sigma}{\msotimes}$-formula $\phi(\bar{X})$ there is a formula $\phi(\bar{X}) \downharpoonright [0,x]$ where $x$ is a `fresh' variable (i.e. $x \notin \bar{X}$) such that for each $a \in \At(\str{M})$ and $\bar{A} \in \str{M} \upharpoonright [0,a]$,
\[
\str{M} \upharpoonright [0,a] \models \phi(\bar{A}) \Longleftrightarrow \str{M} \models \phi(\bar{A}) \downharpoonright [0,a].
\]

This generalises in the following way.
For each $B \in \str{M}$ the set of elements beneath $B$ form the underlying set of a substructure of $\str{M}$. 
We denote this substructure by $\str{M} \upharpoonright B$. 
Relativising quantifiers, for each $\lang{\Sigma}{\msotimes}$-formula $\phi(\bar{X})$ there is a formula $\phi(\bar{X}) \downharpoonright Y$ where $Y \notin \bar{X}$ such that for each $B \in \str{M}$ and each $\bar{A} \in \str{M} \upharpoonright B$,
\[
\str{M} \upharpoonright B \models \phi(\bar{A}) \Longleftrightarrow \str{M} \models \phi(\bar{A}) \downharpoonright B.
\]
\end{lemma}
\begin{proof}
This is a straightforward application of relativisation of quantifiers. 
\end{proof}

\begin{definition}\label{Def-BoundedTypeSpace}
We write $S_k$ for the set of equivalence classes of the equivalence relation $\approx_k$ which contain a structure of the form $\MSO(w)$ where $w \in \Sigma^*$.
\end{definition}

\begin{lemma}
For each $k \in \bb{N}$ the set $S_k$ is closed under $\lift{\msotimes}$.
Moreover for each $k \in \bb{N}$ the structure $(S_k,\lift{\msotimes})$ is a monoid.
\end{lemma}
\begin{proof}
Suppose that $\str{M}$ and $\str{N}$ are $\lang{\Sigma}{\msotimes}$-structures such that $\str{M} \approx_k \MSO(w)$ and $\str{N} \approx_k \MSO(v)$ for some $w,v \in \Sigma^*$.
Then  by \cref{Prop-kequivcongruence} it follows that $\str{M} \msotimes \str{N} \approx_k \MSO(w) \msotimes \MSO(v) \cong \MSO(w \conc v)$, hence $S_k$ is indeed closed under $\lift{\msotimes}$. 
For each $k \in \bb{N}$ the operation $\lift\msotimes$ on $S_k$ is associative and admits the equivalence class containing $\MSO(\ew)$ as an identity element, hence $(S_k,\lift{\msotimes})$ is a monoid.
\end{proof}

\begin{notation}
For each $k \in \bb{N}$, we will write $\msotimes_k$ for the restriction of $\lift{\msotimes}$ to $S_k$.
\end{notation}

\begin{theorem}\label{Theorem-TMSigmaCorrect}
The theory $T_{\MSO(\Sigma^*)}$ is an axiomatisation of the pseudofinite monadic second order theory of $\Sigma$-words, $\bigcap_{w \in \Sigma^*} \Th(\MSO(w))$.
\end{theorem}
\begin{proof}
It is clear that for each $w \in \Sigma^*$, $\MSO(w) \models T_{\MSO(\Sigma^*)}$, and hence $T_{\MSO(\Sigma^*)} \subseteq \Th(\MSO(\Sigma^*))$. So we must show that if $\str{M}$ is an $\lang{\Sigma}{\msotimes}$-structure and $\str{M} \models T_{\MSO(\Sigma^*)}$ then $\str{M} \models \Th(\MSO(\Sigma^*))$.
By \cref{Lemma-ClassTheory}, $\str{M} \models \Th(\MSO(\Sigma^*))$ is equivalent to $\Th(\str{M}) \subseteq \bigcup_{w \in \Sigma^*} \Th(\MSO(w))$.

Fix $\str{M} \models T_{\MSO(\Sigma^*)}$. 
It is enough to prove that for each $k \in \bb{N}$ there exists $w_k \in \Sigma^*$ (depending on $\str{M}$) such that $\str{M} \approx_k \MSO(w_k)$. Then $\Th(\str{M}) \subseteq \bigcup_{k \in \bb{N}} \Th(\MSO(w_k)) \subseteq  \bigcup_{w \in \Sigma^*} \Th(\MSO(w))$.

For each $k \in \bb{N}$, there is a sentence $\psi_k$ such that for each $\lang{\Sigma}{\msotimes}$-structure $\str{M}$ the following are equivalent,
\begin{enumerate}[-]
\item $\str{M} \models \psi_k$,
\item $\str{M} \approx_k \MSO(w)$ for some $w \in \Sigma^*$.
\end{enumerate}
The sentence $\psi_k$ is given by taking the Hintikka sentences $\phi_1,\ldots,\phi_n$ axiomatising the equivalence classes $C_1,\ldots,C_n$ of $\approx_k$, and forming the disjuntion of those $\phi_i$ such that $C_i$ contains $\MSO(w)$ for some $w \in \Sigma^*$.

Consider the formula $\psi_k \downharpoonright [0,x]$. 
It is clear that for each $\str{M} \models T_{\MSO(\Sigma^*)}$ we have that $\str{M} \upharpoonright [0,0] \cong \MSO(\sigma)$ for some $\sigma \in \Sigma$. 
Clearly $\MSO(\sigma) \models \psi_k$ for each $k \in \bb{N}$, hence $\str{M} \models \psi_k \downharpoonright [0,0]$.
On the other hand, $\str{M} \upharpoonright [0,0^*] = \str{M}$, so that $\str{M} \models \psi_k$ is equivalent to $\str{M} \models \psi_k \downharpoonright [0,0^*]$.

Suppose towards a contradiction that $\str{M} \not\models \psi_k$.
Equivalently, suppose $\str{M} \not\models \psi_k \downharpoonright [0,0^*]$.
By comprehension there is an element of $\str{M}$ which lies above precisely those $a \in \At(\str{M})$ for which $\str{M} \not\models \psi_k \downharpoonright [0,a]$.
This element cannot be $\bot$, as we supposed (towards a contradiction) that it contains $0^*$.
Therefore, by the axioms of $T_{\MSO(\Sigma*)}$, it must contain a smallest atom $a \in \At(\str{M})$ (with respect to $\exle$).
Now since $\str{M} \models \psi_k \downharpoonright [0,0]$ we must have $a \neq 0$. As $a \neq 0$, we can take its immediate predecessor $b \in \At(\str{M})$. 
We then get a contradiction using \cref{Prop-kequivcongruence} and the observation that $\str{M} \upharpoonright [0,a] \cong \str{M} \upharpoonright [0,b] \msotimes \MSO(\sigma)$  where $\sigma$ is the unique element of $\Sigma$ such that $\str{M} \models P_{\sigma}(a)$. 
Since $\str{M} \upharpoonright [0,b] \models \psi_k$ there exists $w \in \Sigma^*$ such that $\str{M} \upharpoonright [0,b] \approx_k \MSO(w)$, and therefore $\str{M} \upharpoonright [0,a] \approx_k \MSO(w')$ where $w' \in \Sigma^*$ is given by appending $\sigma$ to $w$. 
This then implies $\str{M} \upharpoonright [0,a] \models \psi_k$, whence we have arrived at a contradiction.

Therefore $\str{M} \models \psi_k$ as required.
\end{proof}

\begin{remark}
We write $S(T_{\MSO(\Sigma^*)})$ for the space of complete theories extending $T_{\MSO(\Sigma^*)}$, i.e. the space with underlying set comprising theories of the form $\Th(\str{M})$ for some $\str{M} \models T_{\MSO(\Sigma^*)}$ and with topology given by a basis of clopen subsets of the form $\dangle{\phi} \coloneqq \{T \in S(T_{\MSO(\Sigma^*)}): T \models \phi\}$.
For each $k \in \bb{N}$, by \cref{Theorem-TMSigmaCorrect}, $\Th(\str{M}) \mapsto [\str{M}]_{\approx_k}$ is a well-defined function $S(T_{\MSO(\Sigma^*)}) \rightarrow S_k$.

The Fra\"iss\'e-Hintikka theorem tells us that this function is continuous when we equip $S_k$ with the discrete topology. 
This is because it tells us that the preimage of each singleton subset of $S_k$ is axiomatised by a sentence, and hence is basic clopen. 
\end{remark}

\section{The free profinite monoid generated by \texorpdfstring{$\Sigma$}{Σ}}\label{Section-FreeProfin}

\begin{definition}
Let $\Sigma$ be a finite set.
A subset $L$ of $\Sigma^*$ is called a \textbf{recognisable $\Sigma$-language} (or a recognisable subset of $\Sigma^*$) if there exists a finite semigroup $S$, a subset $F \subseteq S$, and a homomorphism $f:\Sigma^* \rightarrow S$ such that $L = f^{-1}(F)$. 
\end{definition}

\begin{remark}
It is well-known that the collection $\Rec(\Sigma)$ of recognisable subsets of $\Sigma^*$ forms a Boolean algebra under the set-theoretic operations of intersection, union, and complement (with the empty set and $\Sigma^*$ being the bottom and top elements respectively). 

Recognisable subsets can also be characterised as the languages accepted by certain finite automata, or as those sets of finite words which are finitely axiomatisable in monadic second order logic (see \cref{Theorem-BuchiMSOTheorem}).

It is also well known that the collection of recognisable languages is closed under the complexification of concatenation $\lift{\cdot}: \pow(\Sigma^*)^2 \rightarrow \pow(\Sigma^*)$ (see \cref{Notation-complexification}).
It is also straightforward to check that the complexification of concatenation is both normal and additive (see \cref{Def-BAlgO}).
It therefore makes sense to talk about the extended Stone dual of $(\Rec(\Sigma),\lift{\conc})$ in the sense of \cref{Subsection-ExtendedStoneDuality}.
\end{remark}

In \cite{GGP} the following theorem is stated, but a proof is not given.

\begin{theorem}[\cite{GGP} Theorem 6.1]\label{Theorem-GGP}
The Boolean algebra $\Rec(\Sigma)$ of recognisable $\Sigma$-languages together with (the complexification of) the binary operation of concatenation is the extended Stone dual of $\profin{\Sigma^*}$, the free profinite monoid on $\Sigma$.
\end{theorem}

Here the Boolean space underlying $\profin{\Sigma^*}$ is viewed together with the graph of the monoid operation which it carries, which is a compatible ternary relation.
It is said in \cite{GGP} that a proof would require advanced machinery from duality theory. 
This theorem can be viewed as a special case of the following considerably more general theorem (by taking $A$ to be $\Sigma^*$).

\begin{theorem}[\cite{GehrkeSTR} Theorem 4.5]
Let $A$ be an abstract algebra. The profinite completion $\profin{A}$ (which is defined in a similar way to the profinite completion of a monoid in \cref{Def-Profinite}) is, up to isomorphism as topological algebras, the extended Stone dual of $\Rec(A)$, the Boolean algebra of recognisable subsets of $A$ equipped with residuation operations of recognisable subsets.
\end{theorem}

Discussion of residuation is outside the scope of this article, and is unnecessary for the specific cases we are interested in (free profinite monoids on finitely many generators).
A full proof of this theorem is presented at length in \cite{GehrkeSTR} alongside other results from duality theory.
The purpose of this section is to offer a proof of the more specific theorem from \cite{GGP}, making use of our model-theoretic analysis of $\Th(\MSO(\Sigma^*))$ to present a concise proof which does not rely on advanced machinery from duality theory.

\begin{definition}\label{Def-Lindenbaum}
Let $T$ be an $\lang{}{}$-theory. 
The \textbf{Lindenbaum-Tarski algebra of $T$}, which we denote by $\LT(T)$, is defined to be the Boolean algebra of equivalence classes of $\lang{}{}$-sentences under the equivalence relation $\sim_T$ given by $\phi \sim_T \psi$ if and only if,
\[
T \models \phi \leftrightarrow \psi.
\]
Note that this tacitly exploits the fact that $\sim_T$ is a congruence (see \cref{Def-Congruence}, the generalisation of which to the setting of universal algebra is straightforward) for the Boolean operations of conjunction, disjunction, and negation.
When working with $\LT(T)$ we  write $\phi$ in place of $[\phi]_{\sim_T}$ to avoid bloated notation.
\end{definition}

The following is a concise (but anachronistic) formulation of a theorem of B\"uchi on the monadic second order logic of finite words. It conceals the automata theoretic machinery used in the original proofs.
It is also worth noting that B\"uchi did not work in the signature $\lang{\Sigma}{\msotimes}$ which we make use of here.

\begin{theorem}[B\"uchi]\label{Theorem-BuchiMSOTheorem}
Let $\Sigma$ be a finite set. 
The map $\FSpec: \LT(T_{\MSO(\Sigma^*)}) \rightarrow \Rec(\Sigma)$ given by $\phi \mapsto \{w \in \Sigma^*: \MSO(w) \models \phi\}$ is an isomorphism of Boolean algebras.
\end{theorem}

In the following proposition we exploit the shorthand for restrictions and relativisations laid out in \cref{Lemma-ResRel}.

\begin{proposition}\label{Prop-SPECiso}
The map $\FSpec$ is still an isomorphism if we enrich the two Boolean algebras as follows,
\begin{enumerate}[(i)]
\item to $\LT(T_{\MSO(\Sigma^*)})$ we append the operation $+$ given by,
\[
\phi + \psi \coloneqq \exists X (\dwcl(X) \wedge \phi \downharpoonright X \wedge \psi \downharpoonright X^{\compl}),
\]
where $\dwcl(X)$ (\textbf{d}own\textbf{w}ards \textbf{cl}osed) is the formula,
\[
\forall y,z ((X(y) \wedge z \exle y) \rightarrow X(z)).
\]
\item to $\Rec(\Sigma)$ we append the operation $\lift{\cnct}$ given by,
\[
L_1 \lift{\cnct} L_2 \coloneqq \{w \cnct v :w \in L_1, v \in L_2\}.
\]
\end{enumerate}
\end{proposition}
\begin{proof}
All that is required, in light of \cref{Theorem-BuchiMSOTheorem}, is to show that for any $\phi,\psi \in \LT$, 
\[
\FSpec(\phi + \psi) = \FSpec(\phi) \lift{\cnct} \FSpec(\psi).
\]
If $w \in \FSpec(\phi + \psi)$ this means that $\MSO(w) \models \phi + \psi$, i.e. there is $A \in \MSO(w)$ such that,
\[
\MSO(w) \models \dwcl(A) \wedge \phi \downharpoonright A \wedge \psi \downharpoonright A^{\compl}.
\]
Therefore we get that $\MSO(w) \upharpoonright A \models \phi$ and $\MSO(w) \upharpoonright A^{\compl} \models \psi$.
Now it is clear, since $\MSO(w) \models \dwcl(A)$, that $\MSO(w) \upharpoonright A \cong \MSO(v_1)$ and $\MSO(w) \upharpoonright A^{\compl} \cong \MSO(v_2)$ for some $v_1,v_2 \in \Sigma^*$ such that $v_1 \cnct v_2 = w$.
As $\MSO(w) \models \phi \downharpoonright A$ is equivalent to $\MSO(v_1) \models \phi$ and $\MSO(w) \models \psi \downharpoonright A^{\compl}$ is equivalent to $\MSO(v_2) \models \psi$, we therefore have $w \in \FSpec(\phi) \lift{\cnct} \FSpec(\psi)$.

For the converse suppose that $w \in \FSpec(\phi) \lift{\cnct} \FSpec(\psi)$. 
Then there exists $v_1 \in \FSpec(\phi)$ and $v_2 \in \FSpec(\psi)$ such that $w = v_1 \cnct v_2$.
Then it is clear that there exists $A \in \MSO(w)$, namely the initial segment corresponding to $v_1$, such that,
\[
\MSO(w) \models \dwcl(A) \wedge \phi \downharpoonright A \wedge \psi \downharpoonright A^{\compl}.
\]
Hence $w \in \FSpec(\phi + \psi)$ as required.
\end{proof}

\Cref{Prop-SPECiso} says that $(\Rec(\Sigma),\lift{\cnct})$ and $(\LT(T_{\MSO(\Sigma^*)}),+)$ differ only superficially.
So towards a proof of \cref{Theorem-GGP} we can consider the extended Stone dual of $(\LT(T_{\MSO(\Sigma^*)}),+)$.

\begin{theorem}
The topological monoid $(S(T_{\MSO(\Sigma^*)}),\msotimes)$ is the extended Stone dual of the Boolean algebra with operator $(\LT(T_{\MSO(\Sigma^*)}),+)$.
\end{theorem}
\begin{proof}
The extended Stone dual of $(\LT(T_{\MSO(\Sigma^*)}),+)$ is, by definition (see \cref{Theorem-ExtendedStoneDuality}), the Boolean space with compatible operation $(S(T_{\MSO(\Sigma^*)}),R_+)$, where $R_+$ is the ternary relation on $S(T_{\MSO(\Sigma^*)})$ given by,
\[
R_+(x,y,z) \Leftrightarrow x \lift{+} y \subseteq z.
\]
It is therefore enough to show that for each $x,y \in S(T_{\MSO(\Sigma^*)})$, $x \msotimes y$ is the unique \ifULTRAMODE{ultrafilter }\fi on $\LT(T_{\MSO(\Sigma^*)})$ such that, $x \lift{+} y \subseteq x \msotimes y$.
In other words, we will show that $R_+$ is the graph of $\msotimes$.

We will in fact show that $\dangle{x \lift{+} y} = x \msotimes y$, where $\dangle{x \lift{+} y}$ is the filter generated by the set $x \lift{+} y$. 
Uniqueness then follows from the fact that there can be no proper inclusions of \ifULTRAMODE{ultrafilter}\fi s of a Boolean algebra\ifULTRAMODE{.}\fi 

Firstly we show $x \lift{+} y \subseteq x \msotimes y$. 
As $x \msotimes y$ is a filter this implies that $\dangle{x \lift{+} y} \subseteq x \msotimes y$.
Let $\phi \in x$ and $\psi \in y$ be given. 
Recall that $\phi + \psi$ is defined as the sentence,
\[
\exists X (\dwcl(X) \wedge \phi \downharpoonright X \wedge \psi \downharpoonright X^{\compl}).
\]
Let $\str{M} \models x$ and $\str{N} \models y$, so that $\str{M} \msotimes \str{N} \models x \msotimes y$.
Then consider the element $A = (\top,\bot) \in \str{M} \msotimes \str{N}$. 
It is immediate that $A^{\compl} = (\bot,\top)$, $\str{M} \msotimes \str{N} \upharpoonright A \cong \str{M}$, and $\str{M} \msotimes \str{N} \upharpoonright A^{\compl} \cong \str{N}$.
Therefore,
\[
\str{M} \msotimes \str{N} \models \dwcl(A) \wedge \phi \downharpoonright A \wedge \psi \downharpoonright A^{\compl}.
\]
So $\str{M} \msotimes \str{N} \models \phi + \psi$, and hence $\phi + \psi \in x \msotimes y$ as required.

To conclude we must show that $x \msotimes y \subseteq \dangle{x \lift{+} y}$. 
We show that each sentence $\eta \in x \msotimes y$ is equivalent over $T_{\MSO(\Sigma^*)}$ to a finite disjunction of sentences contained in $x \lift{+} y$.
Take $\eta \in x \msotimes y$.
Let $\str{M} \models x$ and $\str{N} \models y$, so that $\str{M} \msotimes \str{N} \models \eta$. 
We want to show that there exist $m \in \bb{N}$, $\phi_1,\ldots,\phi_m \in x$, and $\psi_1,\ldots,\psi_m \in y$ such that for $i=1,\ldots,m$,
\[
T_{\MSO(\Sigma^*)} \models \eta \leftrightarrow \bigvee_{i=1}^m (\phi_i + \psi_i).
\] 
Without loss of generality we can assume that $\eta$ is an unnested sentence of quantifier rank $k \in \bb{N}$. 
The equivalence relation $\approx_k$ on the class of $\lang{\Sigma}{\msotimes}$-structures admits finitely many equivalence classes, each of which is axiomatised by an $\lang{\Sigma}{\msotimes}$-sentence. 
Taking the finite disjunction of all sentences of the form $\phi + \psi$, where $\phi \in x$ and $\psi \in y$ each axiomatise an equivalence class of $\approx_k$, which $\msotimes_k$ (see \cref{Def-BoundedTypeSpace}) sends to an equivalence class of $\approx_k$ containing a model of $\eta$. 
This disjunction is equivalent to $\eta$ over $T_{\MSO(\Sigma^*)}$ so we are done. 
\end{proof}

\begin{remark}
By taking $w \approx_k v$ if and only if $\MSO(w) \approx_k \MSO(v)$, for each $w,v \in \Sigma^*$, we may think of $\approx_k$ as an element of $\Con_{\omega}(\Sigma^*)$. 
In \cref{Theorem-EFcofinal} and the remainder of the paper we will switch freely between these two views of $\approx_k$.
\end{remark}

\begin{theorem}\label{Theorem-EFcofinal}
The congruences $\approx_k$ for $k \in \bb{N}$ form a cofinal subset of the poset $(\Con_{\omega}(\Sigma^*),\supseteq)$.
\end{theorem}
\begin{proof}
Let $\rho \in \Con_{\omega}(\Sigma^*)$. 
We must show that there is $k \in \bb{N}$ such that $w \approx_k w'$ implies $w \sim_{\rho} w'$ for each $w,w' \in \Sigma^*$.
Consider the quotient $\sfrac{\Sigma^*}{\rho}$.
For each $[w]_{\rho} \in \sfrac{\Sigma^*}{\rho}$, we have that the preimage $\pi_{\rho}^{-1}([w]_{\rho}) \subseteq \Sigma^*$ is a recognisable language.
We can enumerate these equivalence classes as $L_1,\ldots,L_n \subseteq \Sigma^*$.
Note that these equivalence classes form a partition of $\Sigma^*$.
By \cref{Theorem-BuchiMSOTheorem}, there are $\lang{\Sigma}{\msotimes}$-formulas $\phi_1,\ldots,\phi_n$ which axiomatise each of these regular languages, meaning for each $i$ such that $1 \leq i \leq n$ we have $L_i = \{w \in \Sigma^*: \MSO(w) \models \phi_i\}$.
Without loss of generality we may assume that each of the sentences $\phi_1,\ldots,\phi_n$ are unnested. 
Now let $k = \max_{1 \leq i \leq n} \qd(\phi_i)$.
We claim that $w \approx_k w'$ implies $w \sim_{\rho} w'$.
But this is immediate, for if $w \approx_k w'$ then we have $w  \models \phi_i$ if and only if $w' \models \phi_i$ for each $i$ such that $1 \leq i \leq n$, and hence there is $i$ such that $1 \leq i \leq n$ for which $w,w' \in L_i$. 
\end{proof}

\begin{theorem}\label{Theorem-Sinvlimchain}
The topological monoid $(S(T_{\MSO(\Sigma^*)}),\msotimes)$ is the inverse limit of the directed system of topological monoids given by restricting $F_{\Sigma^*}:\Con_{\omega}(\Sigma^*) \rightarrow \TopMon$ (see \cref{Def-Profinite}) to the cofinal subset $C$ of $(\Con_{\omega}(\Sigma^*),\supseteq)$ comprising the congruences $\approx_k$ for $k \in \bb{N}$.
\end{theorem}
\begin{proof}
For each $k$ we have a morphism (a continuous map which respects the binary operation) $p_k:S(T_{\MSO(\Sigma^*)}) \rightarrow S_k$ given by taking a completion of $T_{\MSO(\Sigma^*)}$ to the $\approx_k$ class of any of its models.
These morphisms commute with the projection maps which make up the directed system (see \cref{Def-Profinite}).

Now suppose $(X,\times)$ is some other topological monoid, and that for each $k \in \bb{N}$ we have a morphism $q_k:X \rightarrow S_k$.
Moreover suppose these morphisms $q_k$ commute with the projection maps. 
We must show that there exists a unique morphism $g:X \rightarrow S(T_{\MSO(\Sigma^*)})$ such that $p_k \circ g = q_k$ for each $k \in \bb{N}$.
Taking $x \in X$, we can form $\bar{q}(x) \coloneqq (q_k(x))_{k \in \bb{N}}$.
Similarly for $x \in S(T_{\MSO(\Sigma^*)})$ we can form $\bar{p}(x) \coloneqq (p_k(x))_{k \in \bb{N}}$.
By a straightforward application of the compactness theorem, for each $x$ in $X$ there exists $g(x) \in S(T_{\MSO(\Sigma^*)})$ such that $p_k \circ g (x)= q_k(x)$ for each $k \in \bb{N}$. 
That $g(x)$ is unique is clear. 

So far we have that $g$ is a well defined function, but we must check that it is also continuous and respects the binary operations.
Let $x,x' \in X$ and consider $g(x \times x')$, we want to show that $g(x \times x') = g(x) \msotimes g(x')$.
For each $k \in \bb{N}$ we have,
\begin{align*}
p_k(g(x \times x')) &= q_k(x \times x')\\
&= q_k(x) \msotimes_k q_k(x'),\\
&= p_k(g(x)) \msotimes_k p_k(g(x')),\\
&= p_k(g(x) \msotimes g(x')).
\end{align*}
But if two elements of $S(T_{\MSO(\Sigma^*)})$ agree with respect to all of the morphisms $p_k$, then they must in fact be equal, hence $g(x \times x') = g(x) \msotimes g(x')$.
The topology on $S(T_{\MSO(\Sigma^*)})$ has a basis of clopens comprising sets of the form $\dangle{\phi} \coloneqq \{x: x \models \phi\}$ where $\phi$ is a sentence. 
Consider $g^{-1}(\dangle{\phi})$.
Taking $k = \qd(\phi)$, there exists a subset $F \subseteq S_k$ such that $\dangle{\phi} = p_k^{-1}(F)$.
Then it follows at once from the fact that $p_k \circ g = q_k$ that $g^{-1}(\dangle{\phi}) = q_k^{-1}(F)$, which is clopen as it is the preimage of a continuous map into a finite discrete space.
Hence $g$ is continuous.

Having verified directly that $(S(T_{\MSO(\Sigma^*)}),\msotimes)$ satisfies the universal property for inverse limits the proof is complete.
\end{proof}

\begin{theorem}
The topological monoids $(S(T_{\MSO(\Sigma^*)}),\msotimes)$ and $\profin{\Sigma^*}$ are isomorphic.
\end{theorem}
\begin{proof}
By \cref{Def-Profinite}, $\profin{\Sigma^*}$ is $\varprojlim F_{\Sigma^*}$, where $F_{\Sigma^*}:\Con_{\omega}(\Sigma^*) \rightarrow \TopMon$ takes a finite index congruence $\rho$ to the quotient $\sfrac{\Sigma^*}{\rho}$.
In \cref{Theorem-EFcofinal} it was shown that the congruences of the form $\approx_k$, for $k \in \bb{N}$, form a cofinal subset $C$ of the poset $(\Con_{\omega}(\Sigma^*),\supseteq)$.
Hence $\profin{\Sigma^*}$ is $\varprojlim F'$.
\Cref{Lemma-CofinalDirectedSystem} then tells us that $\varprojlim F_{\Sigma^*} = \varprojlim F'$ where $F'$ is the restriction of $F_{\Sigma^*}$ to $C$.
Finally, \cref{Theorem-Sinvlimchain} says that $\varprojlim F'$ is $(S(T_{\MSO(\Sigma^*)}),\msotimes)$.
Uniqueness of inverse limits up to isomorphism then allows us to conclude that $(S(T_{\MSO(\Sigma^*)}),\msotimes)$ and $\profin{\Sigma^*}$ are isomorphic as required.
\end{proof}

\begin{corollary}
$\profin{\Sigma^*}$ is the extended Stone dual of $(\Rec(\Sigma),\lift{\conc})$.
\end{corollary}

\printbibliography

\end{document}